\documentclass[10pt]{amsart}
\usepackage{amsmath, latexsym, amsthm, amsfonts,bm,amssymb} 
\usepackage{appendix}
\usepackage{natbib} 
\usepackage{url} 
\usepackage{enumerate}
\usepackage{mathtools} 

\bibpunct{(}{)}{;}{a}{}{,} 

\theoremstyle{plain}
\newtheorem{thm}{Theorem}
\newtheorem{prop}{Proposition}
\newtheorem{lemma}{Lemma}

\newtheorem{defin}{Definition}
\newtheorem{assump}{Assumption}
\newtheorem{appxlem}{Lemma}[section]

\theoremstyle{remark}
\newtheorem{rem}{Remark}
\newtheorem{exam}{Example}

\def\ex{{\rm {\mathbb E\,}}}

\allowdisplaybreaks

\begin{document}

\title[Non-parametric Bayesian drift estimation]{Non-parametric Bayesian drift estimation for stochastic differential equations}

\author{Shota Gugushvili}
\address{Mathematical Institute\\
Leiden University\\
P.O. Box 9512\\
2300 RA Leiden\\
The Netherlands}
\email{shota.gugushvili@math.leidenuniv.nl}

\author{Peter Spreij}
\address{Korteweg-de Vries Institute for Mathematics\\
Universiteit van Amsterdam\\
PO Box 94248\\
1090 GE Amsterdam\\
The Netherlands}
\email{spreij@uva.nl}

\subjclass[2000]{Primary: 62G20, Secondary: 62M05}

\keywords{Drift coefficient; Non-parametric Bayesian estimation; Posterior consistency; Stochastic differential equation}

\begin{abstract}

We consider non-parametric Bayesian estimation of the drift coefficient of a one-dimensional stochastic differential equation from discrete-time observations on the solution of this equation. Under suitable regularity conditions that are weaker than those previosly suggested in the literature, we establish posterior consistency in this context. Furthermore, we show that posterior consistency extends to the multidimensional setting as well, which, to the best of our knowledge, is a new result in this setting.

\end{abstract}

\date{\today}

\maketitle

\section{Introduction}
\label{intro}
Consider the $d$-dimensional stochastic differential equation
\begin{equation}
\label{sde}
\mathrm{d}X_t=b(X_t)\mathrm{d}t+\mathrm{d}W_t
\end{equation}
driven by a $d$-dimensional Brownian motion $W,$ and assume that it has a unique (in the sense of the probability law) non-exploding weak solution. One can start with a coordinate mapping process $X$ (that is $X_t(\omega)=\omega(t)$) on the canonical space $(\mathcal{C}(\mathbb{R}_{+}),\mathcal{B}(\mathcal{C}(\mathbb{R}_{+})))$ of continuous functions $\omega:\mathbb{R}_{+}\rightarrow\mathbb{R}^d,$ a flow of sigma-fields $\{ \mathcal{F}_t^X \}$ and the $d$-dimensional Wiener measure $Q$ on $(\mathcal{C}(\mathbb{R}_{+}),\mathcal{B}(\mathcal{C}(\mathbb{R}_{+}))),$ and then, as is well-known (see e.g.\ Proposition 3.6 and Remark 3.7 on p.\ 303 in \cite{karatzas}), under suitable conditions on the drift coefficient $b$ and for any fixed initial distribution $\mu$ one can obtain a weak solution  $(X,W),$ $(\mathcal{C}(\mathbb{R}_{+}),\mathcal{F},P^{\mu}_b),$ $\{\mathcal{F}_t\}$ to \eqref{sde} through the Girsanov theorem. The filtration $\{\mathcal{F}_t\}$ can be made to satisfy the usual conditions by suitably augmenting and completing the filtration $\{\mathcal{F}_t^X\},$ cf.\ Remark 3.7 on p.\ 303 in \cite{karatzas}. Henceforth we will assume that we are in this canonical setup. We will also assume that $X$ is ergodic with a unique ergodic distribution $\mu_b$ and is in fact initialised at $\mu_b,$ so that $\mu=\mu_b.$ Furthermore, we will abbreviate $P^{\mu_b}_b$ to $P_b.$

Suppose that the drift coefficient $b=(b_1,\ldots,b_d)$ belongs to some non-parametric class. Denote by $b_0=(b_{0,1},\ldots,b_{0,d})$ the true drift coefficient and assume that corresponding to it a sample $X_0,X_{\Delta},X_{2\Delta},\ldots,X_{n\Delta}$ is given. The goal is to estimate $b_0$ non-parametrically. The problem of non-parametric estimation of $b_0$ from discrete-time observations has received considerable attention in the literature. For frequentist approaches to the problem see for instance \cite{comte}, \cite{reiss} and \cite{jacod} in the one-dimensional case ($d=1$) and \cite{dalalyan} and \cite{schmisser} in the general multidimensional case ($d\geq 1$). However, a non-parametric Bayesian approach to estimation of $b_0$ is also possible, see e.g.\ \cite{schauer}, \cite{meulen} and \cite{zanten12}. In particular, under appropriate assumptions on the drift coefficient $b,$ the weak solution to \eqref{sde} will admit transition densities $p_b(t,x,y),$ and employing the Markov property, the likelihood corresponding to the observations $X_{i\Delta}$'s can be written as
\begin{equation}
\label{likelihood}
\pi_b(X_0)\prod_{i=1}^n p_b(\Delta,X_{(i-1)\Delta},X_{i\Delta}),
\end{equation}
where $\pi_b$ denotes a density of the distribution $\mu_b$ of $X_0$ (under our conditions $\pi_b$ and $p_b$ will be strictly positive and finite, see Sections \ref{main} and \ref{main2} for details). A Bayesian would put a prior $\Pi$ on the class of drift coefficients, say $\mathcal{X},$ and obtain a posterior measure of any measurable set $B\subset\mathcal{X}$ through Bayes' formula
\begin{equation}
\label{bayes}
\Pi(B|X_0,\ldots,X_{n\Delta})=\frac{\int_B \pi_b(X_0)\prod_{i=1}^n p_b(\Delta,X_{(i-1)\Delta},X_{i\Delta}) \Pi(\mathrm{d}b)}{ \int_{\mathcal{X}} \pi_b(X_0)\prod_{i=1}^n p_b(\Delta,X_{(i-1)\Delta},X_{i\Delta}) \Pi(\mathrm{d}b) }.
\end{equation}
Here we tacitly assume suitable measurability of the integrands, so that the integrals in \eqref{bayes} are well-defined. In the Bayesian paradigm, posterior encapsulates all the information required for inferential purposes. Once posterior is available, one can proceed with computation of Bayes point estimates, credible sets and other quantities of interest in Bayesian statistics.

It has been argued convincingly in \cite{diaconis} and elsewhere that a desirable property of a Bayes procedure is posterior consistency. In our context this will mean that for every neighbourhood (in a suitable topology) $U_{b_0}$ of $b_0,$
\begin{equation*}
\Pi(U_{b_0}^c|X_0,\ldots,X_{n\Delta})\rightarrow 0, \quad  \text{${P}_{b_0}$-a.s.}
\end{equation*}
as $n\rightarrow\infty$ (see Sections \ref{main} and \ref{main2} for details). That is, roughly speaking, a consistent Bayesian procedure asymptotically puts posterior mass equal to one on every fixed neighbourhood of the true parameter: the posterior concentrates around the true parameter. In an infinite-dimensional setting, such as the one we are dealing with, posterior consistency is a subtle property that depends in an essential way on a specification of the prior, see e.g.\ \cite{diaconis}. Note also that the notion of posterior consistency depends on the topology on  $\mathcal{X}.$ Ideally one would like to establish posterior consistency in strong topologies. An implication of posterior consistency is that even though two Bayesians might start with two different priors, the role of the prior in their inferential conclusions will asymptotically, with the sample size growing indefinitely, wash out, and the two will eventually agree. Furthermore, posterior consistency also implies that the centre (in an appropriate sense) of the posterior distribution is a consistent (in the frequentist sense) estimator of the true parameter. For an introductory treatment of posterior consistency see \cite{wasserman}.

In the context of discretely observed scalar diffusion processes given as solutions to stochastic differential equations, posterior consistency has been recently addressed in \cite{meulen}, while the case when a continuous record of observations from a scalar diffusion process is avaiable was covered under various setups in \cite{meulen0}, \cite{panzar} and \cite{
pokern}, where in particular the contraction rates of the posterior were derived. The techniques used in the latter three papers are of little use in the case of discrete observations. The proof of posterior consistency in \cite{meulen} is based on the use of martingale arguments in a fashion similar to \cite{tang}, see also \cite{ghosal06}. The latter paper deals with posterior consistency for estimation of the transition density of an ergodic Markov process. The idea of using martingale arguments in the proofs of consistency of nonparametric Bayesian procedures goes back to \cite{walker0} and \cite{walker} in the i.i.d.\ setting. On the other hand, a similarity between the arguments used in the proof of posterior consistency in \cite{tang} and \cite{meulen} is to a considerable extent on a conceptual level only: conditions for posterior consistency in \cite{tang} involve conditions on transition densities that typically cannot be transformed into conditions on the drift coefficients, because transition densities associated with stochastic differential equations are usually unknown in explicit form. Furthermore, in the setting of \cite{meulen}, who deal with ergodic and strictly stationary scalar diffusion processes (in particular, $X_0$ is initialised at the ergodic distribution of the process $X$), one cannot assume that the density $\pi_{b_0}$ of $X_0$ is known (as done on p.\ 1714 in \cite{tang}), for that would completely determine the unknown drift coefficient $b_0.$

The assumption on the class of drift coefficients in Theorem 3.5 of \cite{meulen} (the latter deals with posterior consistency), namely uniform boundedness of the drift coefficients, is quite restrictive in that it excludes even such a prototypical example of a stochastic differential equation as the Langevin equation (here we assume $d=1$)
\begin{equation}
\label{ou}
\mathrm{d}X_t=-\beta X_t \mathrm{d}t+ \sigma \mathrm{d}W_t,\\
\end{equation}
where $\beta$ and $\sigma$ are two constants. A solution to \eqref{ou} is called an Ornstein-Uhlenbeck process, see Example 6.8 on p.~358 in \cite{karatzas} and p.\ 397 there. Hence, there is room for improvement.

In this work we will show that under suitable conditions posterior consistency in the one-dimensional case still holds for the class of unbounded drift coefficients satisfying the linear growth condition. In particular, the case of the Langevin equation is covered. In our proof of posterior consistency we follow the same train of thought as initiated in \cite{walker0} and \cite{walker}, at the same time making use of ideas from \cite{tang} and especially from \cite{meulen}. According to \cite{meulen}, p.\ 51, the boundedness condition on the drift coefficients cannot be avoided in their approach due to technical reasons. Our analysis and contribution to the literature, however, shows that given a willingness to assume some reasonable and classical conditions on the drift coefficients, the case of unbounded drift coefficients can also be covered via techniques similar to those in \cite{meulen}. Perhaps more importantly, under some extra, but standard assumptions in non-parametric inference for multidimensional stochastic differential equations (cf.\ \cite{dalalyan} and \cite{schmisser}), we show that our analysis in the one-dimensional case extends to the multidimensional setting as well. To the best of our knowledge, this is a new result in this context.

The rest of the paper is organised as follows: in the next section we state our main result in the one-dimensional case, while Section \ref{main2} deals with the general multidimensional case. In Section \ref{discussion} we provide a brief discussion on the obtained results. The proofs of the results from Sections \ref{main} and \ref{main2} are given in Section \ref{proofs}. Finally, Appendices \ref{appendix} and \ref{appendixB} contain several auxiliary statements used in Section \ref{proofs} together with their proofs.

\section{Posterior consistency: one-dimensional case}
\label{main}

In this section we consider the one-dimensional case ($d=1$). The class of drift coefficients we will be looking at will be a subset of the class $\widetilde{\mathcal{X}}(K)$ introduced below.

\begin{defin}
\label{driftclass}
The family $\widetilde{\mathcal{X}}(K)$ consists of Borel-measurable drift coefficients $b:\mathbb{R}\rightarrow\mathbb{R}$ possessing the following two properties:
\begin{enumerate}[(a)]
\item for some constant $K>0$ and $\forall b\in\widetilde{\mathcal{X}}(K),$ the linear growth condition
\begin{equation*}
|b(x)|\leq K (1+|x|)
\end{equation*}
is satisfied, and
\item for each $b\in\widetilde{\mathcal{X}}(K)$ there exist two constants $r_b>0$ and $M_b>0,$ such that
\begin{equation*}
b(x) \operatorname{sgn}(x) \leq -r_b, \quad \forall |x|\geq M_b
\end{equation*}
holds.
\end{enumerate}
\end{defin}

\begin{rem}
\label{driftrem}
Analogously to considering $L_p$-spaces instead of
$\mathcal{L}_p$-spaces, we will identify two functions $b_1$ and $b_2$
in $\widetilde{X}(K),$ if $b_1=b_2$ Lebesgue a.e.\ \qed
\end{rem}

\begin{rem}
\label{remdriftclass1}
The class $\widetilde{\mathcal{X}}(K)$ is such that the case of the Langevin equation \eqref{ou} with $\sigma=1$ is covered for parameter $\beta$ ranging in the interval $(0,K].$ \qed
\end{rem}

The goal of the following proposition is to show that when $b\in\widetilde{\mathcal{X}}(K),$ a unique non-exploding weak solution to \eqref{sde} exists and has certain desirable properties. Although, strictly speaking, a weak solution is a triple $(X,W),$ $(\Omega,\mathcal{F},P^{\mu}_b),$ $\{\mathcal{F}_t\},$ in order to avoid cumbersome formulations, in the sequel we will at times take a liberty to call $X$ itself a weak solution. 

\begin{prop}
\label{classchi}
For each $b\in\widetilde{\mathcal{X}}(K),$ where $\widetilde{\mathcal{X}}(K)$ is defined in Definition \ref{driftclass},
\begin{enumerate}[(1)]
\item a unique non-exploding weak solution to \eqref{sde} exists,
\item the weak solution $X$ to \eqref{sde} is ergodic with unique ergodic distribution $\mu_b$ admitting a density $0<\pi_b(x)<\infty,x\in\mathbb{R}$ with respect to the Lebesgue measure, and
\item transition probabilities $P_b(t,x,\cdot)$ are absolutely continuous with respect to the Lebesgue measure with densities $0<p_b(t,x,y)<\infty,(t,x,y)\in(0,\infty)\times\mathbb{R}\times\mathbb{R}.$
\end{enumerate}
\end{prop}

\begin{rem}
\label{remdriftclass2}
The contents of Proposition \ref{classchi} are standard, but perhaps not available at one place in the literature. The linear growth condition (a) in Definition \ref{driftclass} is a standard assumption to ensure existence of a unique non-exploding weak solution to \eqref{sde}, see e.g.\ Proposition 3.6 and Remark 3.7 on p.\ 303, Theorem 5.15 on p.~341 and Remark 5.19 on p.~342 in \cite{karatzas}. In fact this assumption allows one to construct a weak solution via the Girsanov theorem as mentioned in the beginning of Section \ref{intro}. Property (b) in Definition \ref{driftclass} is a classical assumption ensuring existence of a unique ergodic distribution for $X,$ see e.g.\ Assumption $(H^*)$ on p.\ 548 in \cite{flor}. Finally, the two properties in Definition \ref{driftclass} also yield a short proof of part (3) of Proposition \ref{classchi}. See Section \ref{proofs} for more details. \qed
\end{rem}

\begin{rem}
\label{remdriftclass3}
Apart of its use in the proof of Proposition \ref{classchi}, the linear growth condition (a) on the drift coefficients in Definition \ref{driftclass} is also used e.g.\ when establishing formula \eqref{expmartingale} in the proof of Lemma \ref{lemma6.1} in Appendix \ref{appendix}, and more generally in those instances where we invoke the Girsanov theorem. Furthermore, property (b) from Definition \ref{driftclass} ensures that for $b\in\widetilde{\mathcal{X}}(K)$ the density $\pi_b$ of the ergodic distribution $\mu_b$ of $X$ decays exponentially fast at infinity, cf.\ formula \eqref{expfast}. Hence $\mu_b$ has moments of all orders and for any $b_1,b_2\in\widetilde{\mathcal{X}}(K),$ the Kullback-Leibler divergence ${\rm{K}}(\mu_{b_1},\mu_{b_2})=\int_{\mathbb{R}}\pi_{b_1}(x)\log(\pi_{b_1}(x)/\pi_{b_2}(x))\mathrm{d}x$ is finite. This comes in handy in the proof of Lemma \ref{lemma5.1} in Appendix \ref{appendix}. \qed
\end{rem}

\begin{rem}
\label{remlikelihood}
Positivity of $\pi_b$ and $p_b$ formally justifies rewriting the likelihood as in \eqref{likelihood} and allows us to employ the likelihood ratio $L_n(b)$ in the proof of Theorem \ref{mainthm}.\qed
\end{rem}

\begin{rem}
\label{remlikelihood2}
Measurability of the mapping $ b \mapsto p_b(t,x,y)$ is a subtle property essential in \eqref{bayes}, but it is difficult to ascertain it in a general setting. Therefore we will simply tacitly assume that all the quantities in \eqref{bayes} (or in other formulae where we integrate with respect to the prior) are suitably measurable. \qed
\end{rem}

Since the notion of posterior consistency depends on a topology on the class of drift coefficients under consideration, we have to introduce the latter first. We will base our topology on the transition operators $P_{\Delta}^b.$ Transition operators associated with \eqref{sde} and acting on the class of bounded measurable functions $f:\mathbb{R}\rightarrow\mathbb{R}$ are defined by
\begin{equation*}
P_t^bf(x)=\int_{\mathbb{R}}p_b(t,x,y)f(y)\mathrm{d}y.
\end{equation*}
We want our topology to separate distinct drift coefficients, which can be thought of as an identifiability condition. At the same time we want the posterior measure to concentrate on arbitrarily small neighbourhoods of the true parameter $b_0.$ Fortunately, this will be possible with our choice of topology, as it will have the required separation property. 

As it often happens in practice, it will be convenient in our case to define a topology not by directly specifying the open sets, but rather by specifying a subbase $\widetilde{\mathcal{U}}$ (for a notion of a subbase see e.g.\ p.\ 37 in \cite{dudley}).

\begin{defin}
\label{Ub}
Let $\nu$ be a finite Borel measure on $\mathbb{R}$ that assigns strictly positive mass to every non-empty open subset of $\mathbb{R},$ and let $\mathcal{C}_{bdd}(\mathbb{R})$ denote the class of all bounded continuous functions on $\mathbb{R}.$ For fixed $b\in\widetilde{\mathcal{X}}(K),f\in\mathcal{C}_{bdd}(\mathbb{R})$ and $\varepsilon>0,$ define
\begin{equation*}
U_{f,\varepsilon}^b=\{\widetilde{b}\in{\widetilde{\mathcal{X}}}(K):\| P_{\Delta}^{\widetilde{b}}f - P_{\Delta}^{{b}}f \|_{1,\nu}<\varepsilon\}.
\end{equation*}
Here $\|\cdot\|_{1,\nu}$ denotes the $L_1$-norm with respect to the measure $\nu.$
\end{defin}

The following definition specifies a topology on $\widetilde{\mathcal{X}}(K).$

\begin{defin}
\label{topology}
The topology $\widetilde{\mathcal{T}}$ on $\widetilde{\mathcal{X}}(K)$ is determined by the requirement that the family
\begin{equation*}
\widetilde{\mathcal{U}}=\{ U_{f,\varepsilon}^b:f\in\mathcal{C}_{bdd}(\mathbb{R}),\varepsilon>0,b\in\widetilde{\mathcal{X}}(K) \}
\end{equation*}
is a subbase for $\widetilde{\mathcal{T}}.$
\end{defin}

\begin{rem}
\label{remtopology-1}
The topology in Definition \ref{topology} clearly depends on the choice of the measure $\nu,$ but since $\nu$ is assumed to be fixed beforehand and its specific choice is not of great importance for subsequent developments, it is not reflected in our notation. \qed
\end{rem}

\begin{rem}
\label{remtopology0} The fact that Definition \ref{topology} is a valid definition follows from a standard result in general topology, Theorem 2.2.6 in \cite{dudley}. \qed
\end{rem}

\begin{rem}
\label{remtopology}
The topology in Definition \ref{topology} has already been employed in \cite{meulen}, who in that respect follow Section 6 in \cite{tang}. For a $\mathcal{C}^2$-function $f$ and a small $\Delta,$
\begin{equation*}
P_t^{b}f(x)-P_t^{\widetilde{b}}f(x)\approx \Delta (b(x)-\widetilde{b}(x)) f^{\prime}(x),
\end{equation*}
cf.\ p.\ 50 in \cite{meulen}. Hence for a small $\Delta,$ the topology $\widetilde{\mathcal{T}}$ in some sense resembles the topology induced by the $L_1(\nu)$-norm on $\widetilde{\mathcal{X}}(K).$ \qed
\end{rem}

In Lemma \ref{lemma3.2} given below we will show that the topology of Definition \ref{topology} has the Hausdorff property. This is perfectly sufficient for our purposes. For a notion of a Hausdorff space see e.g.\ p.\ 30 in \cite{dudley}.

\begin{lemma}
\label{lemma3.2}
The topological space $(\widetilde{\mathcal{X}}(K),\widetilde{\mathcal{T}})$ with $\widetilde{\mathcal{T}}$ as in Definition \ref{topology} is a Hausdorff space.
\end{lemma}

We are ready to give the definition of posterior consistency used in the present work. For a definition of a neighbourhood used in it, see p.\ 26 in \cite{dudley}.

\begin{defin}
\label{defconsistency}
Let the prior $\Pi$ be defined on a set $\mathcal{X}(K)\subset\widetilde{\mathcal{X}}(K)$ and let $b_0\in\mathcal{X}(K).$ We say that posterior consistency holds at $b_0,$ if for every neighbourhood $U_{b_0}$ of $b_0$ in the relative topology ${\mathcal{T}}=\{ A \cap \mathcal{X}(K):A\in\widetilde{\mathcal{T}} \}$ (with $\widetilde{\mathcal{T}}$ as in Definition \ref{topology}) we have
\begin{equation*}
\Pi(U_{b_0}^c|X_0,\ldots,X_{n\Delta}) \rightarrow 0, \quad \text{$P_{b_0}$-a.s.}
\end{equation*}
as $n\rightarrow\infty.$
\end{defin}

We need yet another definition (the definition of a uniformly equicontinuous family of functions appearing in it can be found on p.\ 51 in \cite{dudley}).

\begin{defin}
\label{equicontinuity}
A family $\mathfrak{F}$ of functions $f:\mathbb{R}\rightarrow\mathbb{R}$ is called locally uniformly equicontinuous, if for any compact set $F\subset\mathbb{R},$ the restrictions $f|_{F}$ of the functions $f\in\mathfrak{F}$ to $F$ form a uniformly equicontinuous family of functions; i.e., for every $\varepsilon>0,$ there exists a $\delta>0,$ such that the inequality
\begin{equation*}
\sup_{f\in\mathfrak{F}}\sup_{\substack{ x,y\in F \\|x-y|<\delta}} |f(x)-f(y)|<\varepsilon
\end{equation*}
holds.
\end{defin}

The following will be a collection of drift coefficients we will be looking at in our first main result, Theorem \ref{mainthm}.

\begin{defin}
\label{chik}
Let $\mathcal{X}(K)$ be the collection of drift coefficients, such that $\mathcal{X}(K)\subset\widetilde{\mathcal{X}}(K)$ and $\mathcal{X}(K)$ is a locally uniformly equicontinuous family of functions.
\end{defin}

\begin{rem}
\label{remequicontinuity}
Functions $f$ belonging to some locally uniformly equicontinuous family $\mathfrak{F}$ of functions  are obviously continuous. If the family $\mathfrak{F}$ is such that for every compact set $F\subset\mathbb{R}$ the restrictions $f|_{F}$ of the functions $f\in\mathfrak{F}$ to $F$ uniformly satisfy a H\"older condition (i.e.\ a H\"older condition with the same H\"older constants), then $\mathfrak{F}$ is a locally uniformly equicontinuous family of functions. \qed
\end{rem}

We summarise our assumptions.

\begin{assump}
\label{standing}
Assume that
\begin{enumerate}[(a)]
\item a unique in law non-exploding weak solution to \eqref{sde} corresponding to each $b\in\mathcal{X}(K)$ is initialised at the ergodic distribution $\mu_{b},$
\item $b_0\in\mathcal{X}(K)$ denotes the true drift coefficient,
\item a discrete-time sample $X_0,\ldots,X_{n\Delta}$ from the solution to \eqref{sde} corresponding to $b_0$ is available (we assume that we are in the canonical setup as in Section \ref{intro}), and finally, $\Delta$ is fixed and independent of $n.$
\end{enumerate}
\end{assump}

The following is our first main result.

\begin{thm}
\label{mainthm}
Let Assumption \ref{standing} hold and suppose that the prior $\Pi$ on $\mathcal{X}(K)$ is such that
\begin{equation}
\label{priorcondition}
\Pi(b\in{\mathcal{X}}(K):\|b-b_0\|_{2,\mu_{b_0}}<\varepsilon)>0, \quad \forall \varepsilon>0.
\end{equation}
Here $\|\cdot\|_{2,\mu_{b_0}}$ denotes the $L_2$-norm with respect to measure $\mu_{b_0}.$ Then posterior consistency as in Definition \ref{defconsistency} holds.
\end{thm}

\begin{rem}
\label{rem_prior2}
The fact that the members $b$ of the parameter set $\mathcal{X}(K)$ must satisfy the linear growth assumption for a uniform constant $K,$ as well as the fact that $\mathcal{X}(K)$ must be a locally uniformly equicontinuous family of functions, is unfortunate, as this excludes many interesting and popular priors in non-parametric Bayesian statistics (for instance the Gaussian process priors; cf.\ \cite{panzar}), but cannot be avoided with the current method of proof (cf.\ the remarks on p.\ 51 and p.\ 60 in \cite{meulen}). In fact, already in the parametric setting stronger conditions are used to theoretically justify validity of Bayesian computational approaches, such as the ones in \cite{beskos} and \cite{eraker}. We also remark that in the parametric estimation case from discrete-time observations, some domination conditions on the drift coefficients are still imposed in the asymptotic studies in the frequentist literature, that do not appear to be easily dispensable, except perhaps in simple cases like that of the Langevin equation \eqref{ou}; see e.g.\ \cite{dacunha} and \cite{flor}. \qed
\end{rem}

\begin{rem}
\label{rem_prior1}
Condition \eqref{priorcondition} on the prior $\Pi$ is formulated in terms of the $L_{2}(\mu_{b_0})$-neighbourhoods, while the posterior consistency assertion returned by Theorem \ref{mainthm} is for the weak topology $\mathcal{T}.$ However, by Remark \ref{remtopology}, for a small $\Delta$ the `discrepancy' is not as dramatic as it may seem at the first sight. \qed
\end{rem}

\begin{rem}
\label{prior_exam1}
Since $b_0$ is unknown, the prior $\Pi$ must verify \eqref{priorcondition} at all parameter values $b\in\mathcal{X}.$ Such priors do exist: since conditions of Theorem \ref{mainthm} are implied by conditions in Theorem 3.5 in \cite{meulen} (an assumption on the drift coefficients $b$ ensuring ergodicity of $X$ is not made explicit in \cite{meulen}, see p.\ 47 there, but this does not cause any problems in our setting), two concrete examples of the prior $\Pi$ can be found in Section 4 in \cite{meulen}. \qed
\end{rem}

\begin{rem}
\label{priorrem}
The conditions of Theorem \ref{mainthm} cover the case of the Langevin equation \eqref{ou} with $\sigma=1$ known and $\beta$ an unknown parameter of interest. \qed
\end{rem}

\section{Posterior consistency: multidimensional case}
\label{main2}

In this section we turn our attention from the one-dimensional case to the general multidimensional case ($d\geq 1$). The developments in this section are parallel to those in Section \ref{main} and involve some repetitions, so we try to be relatively brief.

Our parameter set will be a subset of the class $\widetilde{\mathcal{X}}(K_1,K_2)$ of drift coefficients introduced below.

\begin{defin}
\label{drift_class_md}
The family $\widetilde{\mathcal{X}}(K_1,K_2)$ consists of Borel-measurable drift coefficients $b:\mathbb{R}^d\rightarrow\mathbb{R}^d$ possessing the following three properties:
\begin{enumerate}[(a)]
\item for any $b\in\widetilde{\mathcal{X}}(K_1,K_2),$ there exists a $\mathcal{C}^3$-function $V_b:\mathbb{R}^d\rightarrow\mathbb{R},$ such that
\begin{equation*}
\quad C_b=\int_{\mathbb{R}^d}e^{-2V_b(u)}du<\infty,
\end{equation*}
$|V_b(x)|$ grows not faster than a polynomial of $\|x\|$ at infinity and $b=-[\nabla V_b]^{tr},$ where $\nabla V_b$ is the gradient of $V_b$ and $tr$ denotes transposition;
\item for any $b\in\widetilde{\mathcal{X}}(K_1,K_2),$ there exist three constants $r_b>0,$ $M_b>0$ and $\alpha_b \geq 1,$ such that
\begin{equation*}
b(x) \cdot x \leq - r_b \|x\|^{\alpha_b}, \quad \forall \|x\| \geq M_b,
\end{equation*}
where by dot we denote the usual scalar product on $\mathbb{R}^d$ and  $\|x\|$ is the $L_2$-norm of a vector $x\in\mathbb{R}^d;$
\item there exist two constants $K_1>0$ and $K_2>0,$ such that for any $b\in\widetilde{\mathcal{X}}(K_1,K_2),$
\begin{equation*}
\|b(x)\|\leq K_1(1+\|x\|), \quad \left|\frac{\partial}{\partial x_j} b_i(x) \right | \leq K_2, \quad \forall x\in\mathbb{R}^d, \quad i,j=1,\ldots,d.
\end{equation*}
\end{enumerate}
\end{defin}

\begin{rem}
\label{rem_drift_md}
Assumptions made in Definition \ref{drift_class_md} are more than enough to guarantee existence of the unique (in the sense of the probability law) solution to \eqref{sde}. By Proposition 1 in \cite{schmisser}, cf.\ p.\ 27 in \cite{dalalyan}, these assumptions also imply existence of the unique ergodic distribution $\mu_b$ that has the density
\begin{equation*}
\pi_b(x)=\frac{1}{C_b} e^{-2V(x)}>0
\end{equation*}
with respect to the $d$-dimensional Lebesgue measure. In models in physics the function $V_b$ has the interpretation of the potential energy of the system. Furthermore, Proposition 1.2 in \cite{gobet01} implies existence of strictly positive transition densities $p_b(t,x,y)$ associated with \eqref{sde}. Finally, for any $b,\widetilde{b}\in\mathcal{X}(K_1,K_2)$ we also have that the Kullback-Leibler divergence $\operatorname{K}(\mu_{b},\mu_{\widetilde{b}})$ is finite, which we use in the proof of Lemma \ref{lemma5.1b} in Appendix \ref{appendixB}. \qed
\end{rem}

\begin{rem}
\label{rem_comparison}
Compared to the case $d=1$ in Section \ref{main}, assumptions made in Definition \ref{drift_class_md} on the class of drift coefficients, when specified to the case $d=1,$ are somewhat stronger. \qed
\end{rem}

\begin{rem}
\label{rem_examples_md}
Examples of multidimensional stochastic differential equations satisfying assumptions in Definition \ref{drift_class_md} are given in Section 5.2 in \cite{schmisser}. \qed
\end{rem}

Define the transition operators associated with \eqref{sde} and acting on the class of bounded measurable functions $f:\mathbb{R}^d\rightarrow\mathbb{R}$ by
\begin{equation*}
P_t^bf(x)=\int_{\mathbb{R}^d}p_b(t,x,y)f(y)\mathrm{d}y,
\end{equation*}
and for every fixed $f\in\mathcal{C}_{bdd}(\mathbb{R}^d), \varepsilon>0, b\in\widetilde{\mathcal{X}}(K_1,K_2)$ let 
\begin{equation*}
U_{f,\varepsilon}^b=\{\widetilde{b}\in{\widetilde{\mathcal{X}}}(K_1,K_2):\| P_{\Delta}^{\widetilde{b}}f - P_{\Delta}^{{b}}f \|_{1,\nu}<\varepsilon\},
\end{equation*}
where $\nu$ is a fixed finite Borel measure on $\mathbb{R}^d$ that assigns strictly positive mass to every non-empty open subset of $\mathbb{R}^d.$ Define the topology $\widetilde{\mathcal{T}}$ on $\widetilde{\mathcal{X}}(K_1,K_2)$ through a subbase
\begin{equation*}
\widetilde{\mathcal{U}}=\{ U_{f,\varepsilon}^b: f\in\mathcal{C}_{bdd}(\mathbb{R}^d), \varepsilon>0, b\in\widetilde{\mathcal{X}}(K_1,K_2) \}.
\end{equation*}
 In complete analogy to Lemma \ref{lemma3.2}, we have the following result.

\begin{lemma}
\label{lemma3.2b}
The topological space $(\widetilde{\mathcal{X}}(K_1,K_2),\widetilde{\mathcal{T}})$ with $\widetilde{\mathcal{X}}(K_1,K_2)$ as in Assumption \ref{drift_class_md} is a Hausdorff space.
\end{lemma}

Let ${\mathcal{X}}(K_1,K_2)\subseteq\widetilde{\mathcal{X}}(K_1,K_2),$ with the interpretation that ${\mathcal{X}}(K_1,K_2)$ is our parameter set, and let ${\mathcal{T}}=\{ A \cap \mathcal{X}(K_1,K_2):A\in\widetilde{\mathcal{T}} \}$ be the corresponding relative topology on $\mathcal{X}(K_1,K_2).$  If for any neighbourhood $U_{b_0}\in\mathcal{T}$ of $b_0\in \mathcal{X}(K_1,K_2)$ we have
\begin{equation*}
\Pi(U_{b_0}^c|X_0,\ldots,X_{n\Delta}) \rightarrow 0, \quad \text{$P_{b_0}$-a.s.}
\end{equation*}
as $n\rightarrow\infty,$ we will say that posterior consistency holds at $b_0.$

We summarise our assumptions.

\begin{assump}
\label{standing_b}
Assume that
\begin{enumerate}[(a)]
\item a unique in law non-exploding weak solution to \eqref{sde} corresponding to each $b\in\mathcal{X}(K_1,K_2)$ is initialised at the ergodic distribution $\mu_{b},$
\item $b_0\in\mathcal{X}(K_1,K_2)$ denotes the true drift coefficient,
\item a discrete-time sample $X_0,\ldots,X_{n\Delta}$ from the solution to \eqref{sde} corresponding to $b_0$ is available (we assume that we are in the canonical setup as in Section \ref{intro}), and finally, $\Delta$ is fixed and independent of $n.$
\end{enumerate}
\end{assump}

Under Assumption \ref{standing_b}, the following multidimensional analogue of Theorem \ref{mainthm} holds.

\begin{thm}
\label{mainthm_b}
Let Assumption \ref{standing_b} hold and suppose that the prior $\Pi$ on $\mathcal{X}(K_1,K_2)$ is such that
\begin{equation}
\label{priorcondition_b}
\Pi\left(b\in{\mathcal{X}}(K_1,K_2):\left\{\sum_{i=1}^d \|b_i-b_{0,i}\|_{2,\mu_{b_0}}^2\right\}^{1/2}<\varepsilon\right)>0, \quad \forall \varepsilon>0.
\end{equation}
Then posterior consistency holds.
\end{thm}

Condition \eqref{priorcondition_b} on the prior is of the same type as condition \eqref{priorcondition} in Theorem \ref{mainthm}. We provide an example of a prior $\Pi$ satisfying this condition. The construction of $\Pi$ is similar to that in Example 4.1 in \cite{meulen}. Both examples are related to discrete net priors in non-parametric Bayesian inference problems studied in \cite{gho}. The construction is admittedly artificial, but its sole goal is to show existence of a prior satisfying \eqref{priorcondition_b}.

\begin{exam}
\label{exam_prior2}
Let $\mathfrak{F}$ be a collection of $\mathcal{C}^3$-functions $f:\mathbb{R}\rightarrow\mathbb{R},$ such that
\begin{enumerate}[(a)]
\item for some polynomial function $G:\mathbb{R}\rightarrow\mathbb{R}$ and $\forall f\in\mathfrak{F}$ we have
\begin{equation*}
|f(x)|\leq G(x), \quad \forall x\in\mathbb{R}_{+};
\end{equation*}
\item for some constant $K_1>0$ and $\forall f\in\mathfrak{F}$ we have
\begin{equation*}
|f^{\prime}(x)|\leq \frac{K_1}{2}, \quad \forall x\in\mathbb{R}_{+};
\end{equation*}
\item $\forall f\in\mathfrak{F}$ we have
\begin{equation*}
\int_{\mathbb{R}^d} e^{-2f(\|x\|^2)}dx<\infty;
\end{equation*}
\item $\forall f\in\mathfrak{F}$ there exist two constants $M_f>0$ and $r_f>0,$ such that $f^{\prime}(x)\geq r_f,\forall x\geq M_f;$
\item for some constant $K_2>0$ and $\forall f\in\mathfrak{F},$
\begin{equation*}
\sup_{x\in\mathbb{R}_{+}}\left\{ 4 x |f^{\prime\prime}(x)| + 2 {|f^{\prime}(x)|} \right\}\leq K_2.
\end{equation*}
\end{enumerate}
For all $x\in\mathbb{R}^d$ set $V_f(x)=f(\|x\|^2)$ and $b_f(x)=-[\nabla V_f(x)]^{tr}.$ Let $\mathcal{X}(K_1,K_2)$ be a subset of a collection of all functions $b_f=(b_{f,1},\ldots,b_{f,d})$ obtained in this way (the fact that this is a valid definition, in the sense that the requirements from Definition \ref{drift_class_md} are satisfied, follows by easy, but somewhat tedious computations; note that by taking $f_{\beta}=\beta x/2$ and assuming $d=1$ and $\beta\in (0,K_1],$ we can cover the case of the Langevin equation \eqref{ou}). We get from (b) that for every fixed $i=1,\dots,d,$ the functions $b_{f,i}$ are locally bounded by constants uniform in $f\in\mathfrak{F}$. Furthermore, they are Lipschitz with uniform constants in $f\in\mathfrak{F}$ as well: by the mean value theorem,
\begin{align*}
|b_{f,i}(x)-b_{f,i}(y)|&\leq \| \nabla b_{f,i}(\lambda x + (1-\lambda)y) \| \|x-y\|\\
&\leq \sqrt{d}K_2  \|x-y\|.
\end{align*}
Hence for each $m\in\mathbb{N}$ and $i=1,\ldots,d,$ by the Arzel\`a-Ascoli theorem, see Theorem 2.4.7 in \cite{dudley}, the collection $\mathfrak{B}_{m,i}$ of restrictions $b_{f,i}|_{[-m,m]}$ of the functions $b_{f,i},f\in\mathfrak{F}$ to the intervals $[-m,m]$ is totally bounded for the supremum metric $\|\cdot\|_{\infty}$ (for the required definitions see p.\ 45 and p.\ 52 in \cite{dudley}). Then so is the product $\bigotimes_{i=1}^d \mathfrak{B}_{m,i}$ for the product metric
\begin{equation*}
\|b_f\|_{d,m,\infty}=\max_{i=1,\ldots,d} \| b_{f,i}|_{[-m,m]} \|_{\infty},
\end{equation*}
as well as its subset consisting of elements
\begin{equation*}
b_f|_{[-m,m]^d}=(b_{f,1}|_{[-m,m]},\ldots,b_{f,d}|_{[-m,m]}), \quad f\in\mathfrak{F}.
\end{equation*}
Take a sequence $\epsilon_l\downarrow 0.$ For any $l\in\mathbb{N},$ there exists a finite subset $\mathfrak{F}_{m,\epsilon_l}=\{ f_{n}^{m,\epsilon_l},n=1,\ldots,n_{m,l} \}$ such that for any $f\in\mathfrak{F},$ $\| b_{f} - b_{ f_{n}^{m,\epsilon_l} } \|_{d,m,\infty} < \epsilon_l$ for some $n=1,\ldots,n_{m,l}.$ Let $\widetilde{Q}_1$ and $\widetilde{Q}_2$ be two measures on ${\mathbb{N}},$ such that $q_{j,i}=\widetilde{Q}_i(j)>0,i=1,2,j\in{\mathbb{N}}.$ The prior $\Pi$ on $\mathcal{X}(K_1,K_2)$ is defined by
\begin{equation*}
\Pi=\sum_{m=1}^{\infty}\sum_{l=1}^{\infty}\sum_{n=1}^{n_{m,l}}\frac{q_{m,1}q_{l,2}}{n_{m,l}}\delta_{ b_{ f_{n}^{m,\epsilon_l} } },
\end{equation*}
where $\delta_{ b_{ f_{n}^{m,\epsilon_l} } }$ is the Dirac measure at $b_{ f_{n}^{m,\epsilon_l} }.$
The fact that $\Pi$ satisfies requirement \eqref{priorcondition} of Theorem \ref{mainthm} is the content of Lemma \ref{lemmapriorB} in Section \ref{proofs}. Since $\Pi$ assigns all its mass to a countable subset of $\mathcal{X}(K_1,K_2),$ measurability issues should not concern us  when integrating with respect to $\Pi.$
\qed
\end{exam}

\section{Discussion}
\label{discussion}

In this work we were able to demonstrate that posterior consistency for non-parametric Bayesian estimation of the drift coefficient of a stochstic differential equation holds not only for the class of uniformly bounded drift coefficients and in the one-dimensional setting, as shown previously in \cite{meulen}, but also in the multidimensional setting for the class of drift coefficients satisfying a linear growth assumption. This considerably enlarges the scope of the main result in \cite{meulen}. Interestingly, although derivation of the posterior consistency result in the one-dimensional case is quite involved in \cite{meulen} (cf.\ a remark on p.\ 60 in \cite{meulen}) and replacement of the uniform boundedness condition on the drift coefficient with the linear growth condition requires some technical prowess in the proofs, see in particular the proof of Lemma \ref{lemma6.1} in Appendix \ref{appendix}, generalisation to the multidimensional setting does not involve technicalities far different from those in the one-dimensional setting, provided one suitably restricts the non-parametric class of drift coefficients. In fact, conditions we impose in the multidimensional setting are analogous to those used in the frequentist literature, see \cite{dalalyan} and \cite{schmisser}, which is a comforting fact. On the other hand, posterior consistency results both in \cite{meulen} and in our work are established for a weak topology on the class of drift coefficients. This is a consequence of the fact that we rely on techniques from \cite{walker} in our proofs, which are better suited for proving posterior consistency in weak topologies. Consistency in stronger topologies could have been established and contraction rates of the posterior could have been derived from general results for posterior consistency in Markov chain models had we known existence of certain tests satisfying conditions as in formula (2.2) in \cite{ghosal2}; see Theorem 5 there. Existence of such tests for Markov chain models has been demonstrated in Theorem 3 in \cite{birge}, but unfortunately, the conditions involved in this theorem, cf.\ also formula (4.1) in \cite{ghosal2}, do not appear to hold in general for the stochastic differential equation models such as \cite{meulen} and we are considering. Hence establishing posterior consistency in a stronger topology and derivation of the posterior contraction rate for non-parametric Bayesian drift estimation is an interesting and difficult open problem. A recent paper \cite{pokern}  addresses the latter question for a one-dimensional stochastic differential equation with a periodic drift coefficient. However, this is done under an assumption that an entire sample path $\{X_t:t\in[0,T]\}$ is observed over the time interval $[0,T]$ with $T\rightarrow\infty.$ Moreover, periodic drift coefficients are completely different from the drift coefficients considered in Section \ref{main} of the present work and making use of the techniques from \cite{pokern} is impossible in our setting. Neither are the techniques in \cite{meulen0} and \cite{panzar} of any significant help (these papers deal with continuously observed scalar diffusion processes). It should also be noted that in the frequentist setting too (with $\Delta$ fixed) already in the one-dimensional setting study of convergence rates of non-parametric estimators of the drift and dispersion coefficients is a highly non-trivial task, see e.g.\ \cite{reiss}, where various simplifying assumptions have been made, such as the requirement that the diffusion process under consideration has a compact state space, say $[0,1],$ and is reflecting at the boundary points. Nevertheless, some progress in establishing posterior consistency in a stronger topology than in the present work might be possible in the setting where $\Delta=\Delta_n\rightarrow 0$ in such a way that $n\Delta_n\rightarrow\infty$ (the so-called high-frequency data setting).

Finally, we remark that issues associated with practical implementation of the non-parametric Bayesian approach to estimation of a drift coefficient are outside the scope of the present work. Although much remains to be done in this direction, preliminary studies, such as the ones in \cite{papa} and \cite{schauer}, see also the overview paper \cite{zanten12}, indicate that a non-parametric Bayesian approach in this context is both feasible and leads to reasonable results.

\section{Proofs}
\label{proofs}

\begin{proof}[Proof of Proposition \ref{classchi}]
As already mentioned in Remark \ref{remdriftclass2}, property (a) in Definition \ref{driftclass} suffices to guarantee existence of a unique non-exploding weak solution to equation \eqref{sde}. This proves part (1) of the proposition.

We next prove part (2). Although the result is well-known, a detailed proof does not seem to be available in the literature. We provide it for the reader's convenience. Introduce the scale function
\begin{equation*}
s_b(y)=\int_0^y \exp\left(-2\int_0^z b(x)\mathrm{d}x\right)\mathrm{d}z.
\end{equation*}
To show existence of an ergodic distribution, it is enough to show its existence for a process $\widetilde{X}_t=s_b(X_t),$ cf.\ p.\ 48 in \cite{skorokhod}. This process satisfies the stochastic differential equation
\begin{equation*}
d\widetilde{X}_t=\widetilde{\sigma}(\widetilde{X}_t)\mathrm{d}W_t,
\end{equation*}
where
\begin{equation*}
\widetilde{\sigma}(y)=s_b^{\prime}(s_b^{-1}(y)),
\end{equation*}
see Lemma 9 on p.\ 47 in \cite{skorokhod}. Here $s_b^{-1}$ denotes an inverse of $s_b.$ If we can show that
\begin{equation}
\label{ergcond}
s_b(-\infty)=-\infty, \quad s_b(\infty)=\infty, \quad \int_{\mathbb{R}}\frac{1}{\widetilde{\sigma}^2(y)}\mathrm{d}y <\infty,
\end{equation}
then Theorem 16 on p.\ 51 in \cite{skorokhod} will imply existence of a unique ergodic distribution for $\widetilde{X},$ and hence for $X$ too. However, under our assumptions checking these conditions is easy. Assuming for instance that the first two conditions in \eqref{ergcond} have been verified (the arguments used in their verification are similar to those used in verification of the third one), we will check the last one. By a change of the integration variable $x=s_b^{-1}(y),$ we have
\begin{equation}
\label{nomeri1}
\begin{split}
\int_{\mathbb{R}}\frac{1}{\widetilde{\sigma}^2(y)}\mathrm{d}y&=\int_{\mathbb{R}}\frac{1}{s_b^{\prime}(x)}\mathrm{d}x\\
&=\int_{\mathbb{R}} \exp\left(2\int_0^x b(y)\mathrm{d}y\right)\mathrm{d}x.
\end{split}
\end{equation}
Finiteness of the latter integral can be seen as follows: let $x\geq M_b>0$ and note that
\begin{equation}
\label{expfast}
\begin{split}
\exp\left(2\int_0^x b(y) \mathrm{d}y\right)&=\exp\left(2\int_0^{M_b} b(y) \mathrm{d}y\right)
 \exp\left(2\int_{M_b}^x b(y) \mathrm{d}y\right)\\
& \leq \exp\left(2K\int_0^{M_b} (1+y) \mathrm{d}y\right)
 \exp\left(-2 r_b\int_{M_b}^x \mathrm{d}y\right)\\
&=c_{b,K} \exp(-r_b|x|).
\end{split}
\end{equation}
The argument for negative $x$ is similar and yields a similar inequality. The integral in \eqref{nomeri1} is then finite thanks to the exponential decay property as in \eqref{expfast}, and existence of a unique ergodic distribution follows. By the same Theorem 16 on p.\ 51 in \cite{skorokhod}, the density of the ergodic distribution of $\widetilde{X}$ is given by
\begin{equation*}
\widetilde{\pi}(x)=\frac{(\widetilde{\sigma}(x))^{-2}}{\int_{\mathbb{R}}(\widetilde{\sigma}(x))^{-2}\mathrm{d}x}.
\end{equation*}
By the fact that $s_b$ is strictly increasing and hence $P( X_t \leq x)=P(s_b(X_t) \leq s_b(x)),$ it then follows by differentiation that the density of the ergodic distribution of $X$ is given by
\begin{equation*}
\pi_b(x)=\frac{1}{m_b(\mathbb{R})} \exp\left(2\int_0^x b(y)\mathrm{d}y \right),
\end{equation*}
where
\begin{equation*}
m_b(\mathrm{d}x)=\exp\left(2\int_0^x b(y)\mathrm{d}y\right)\mathrm{d}x
\end{equation*}
is the speed measure of $X.$ Furthermore, by the linear growth condition, $0<\pi_b(x)<\infty,\forall x\in\mathbb{R}.$ This proves part (2).

Finally, for the proof of part (3) we argue as follows: the first two equalities in \eqref{ergcond} and Proposition 5.22 (a) on p.\ 345 in \cite{karatzas} yield that the process $X$ is recurrent. Hence the solution to \eqref{sde} generates a regular diffusion (see Definition 45.2 on p.\ 272 in \cite{williams}). By Theorem 50.11 on pp.\ 294--295 in \cite{williams}, the transition probabilities of $X$ admit continuous, strictly positive and finite densities with respect to the speed measure $m_b(\mathrm{d}y)$
of $X.$ Since from part (2) we have in turn that
\begin{equation*}
0<\exp\left(2\int_0^y b(z)\mathrm{d}z\right)<\infty,\quad \forall z\in\mathbb{R},
\end{equation*}
it follows that transition probabilities of $X$ admit continuous, strictly positive and finite densities $p_b(t,x,y)$ with respect to the Lebesgue measure. This completes the proof of the proposition.
\end{proof}

\begin{proof}[Proof of Lemma \ref{lemma3.2}]
The lemma can be proved by arguments similar to those in the proof of Lemma 3.2 in \cite{meulen}; cf.\ also the proof of Lemma \ref{lemma3.2b} in Appendix \ref{appendixB}. The following result, which is an analogue of Lemma 3.1 in \cite{meulen}, is required in the proof (the arguments from the proof of the latter remain applicable): let $b\in\widetilde{\mathcal{X}}(K).$ Fix $t>0.$ If $b\neq\widetilde{b},$ then $P_t^{b} \neq P_t^{\widetilde{b}}.$
\end{proof}

\begin{proof}[Proof of Theorem \ref{mainthm}]
The proof follows the same main steps as the proof of Theorem 3.5 in \cite{meulen}, which in turn uses some ideas from \cite{tang} and \cite{walker}. In particular, our proof employs Lemmas \ref{lemma5.1}, \ref{lemma5.2} and \ref{lemma5.3} from Appendix \ref{appendix}, that correspond to Lemmas 5.1, 5.2 and 5.3 in \cite{meulen}. Fix $\varepsilon>0,$ take a fixed $f\in\mathcal{C}_{bdd}(\mathbb{R})$ and write
\begin{equation}
\label{setB}
B=\{ b\in\mathcal{X}(K):\| P_{\Delta}^b f - P_{\Delta}^{b_0} f \|_{1,\nu} > \varepsilon \}.
\end{equation}
Without loss of generality we may assume that $\|f\|_{\infty}\leq 1$ and $\varepsilon\leq 2\nu(\mathbb{R}).$ We claim that by the definition of the topology $\mathcal{T}$ it suffices to establish posterior consistency for every fixed $B$ of the above form. Indeed, by intersecting the sets from the base $\widetilde{\mathcal{V}}$ determined by the subbase $\widetilde{\mathcal{U}}$ from Definition \ref{topology} with $\mathcal{X}(K),$ a base for $\mathcal{T}$ can be obtained. Likewise, a subbase $\mathcal{U}$ for $\mathcal{T}$ is obtained by intersecting the sets from the subbase $\widetilde{\mathcal{U}}$ for $\widetilde{\mathcal{T}}$ with $\mathcal{X}(K).$ By definition, an arbitrary neighbourhood $U_{b_0}$ of $b_0$ contains an open set $\hat{U}_{b_0}\in\mathcal{T}.$ The set $\hat{U}_{b_0}$ is a union of open sets $V$ from the base $\mathcal{V},$ $\hat{U}_{b_0}=\bigcup \{V\in\mathcal{V}:V\subset \hat{U}_{b_0} \}.$ There is at least one $V$ that contains $b_0.$ Fix such $V.$ By definition of the subbase $\mathcal{U}$ this set $V$ can be represented as $V=\bigcap_{j=1}^m U^{b_0}_{f_j,\varepsilon_j}$ for some $m,$ positive numbers $\varepsilon_j,$ bounded continuous functions $f_j$ and sets $U^{b_0}_{f_j,\varepsilon_j}$ from the subbase $\mathcal{U}.$ Note that we have 
\begin{equation*}
{U}_{b_0}^c\subset\hat{U}_{b_0}^c \subset V^c=\bigcup_{j=1}^m (U^{b_0}_{f_j,\varepsilon_j})^c.
\end{equation*}
Since
\begin{multline*}
(U^{b_0}_{f_j,\varepsilon_j})^c=\{ b\in\mathcal{X}(K):\|P_{\Delta}^b f_j-P_{\Delta}^{b_0} f_j\|_{1,\nu}\geq\varepsilon_j \}\\
\subset \left\{ b\in\mathcal{X}(K):\|P_{\Delta}^b f_j -P_{\Delta}^{b_0} f_j \|_{1,\nu}>\frac{\varepsilon_j}{2} \right\},
\end{multline*}
say, the claim becomes obvious. 

The posterior measure of a set $B$ given in \eqref{setB} can be written as
\begin{equation*}
\Pi(B|X_0,\ldots,X_{n\Delta})=\frac{\int_{B}L_n(b)\Pi(\mathrm{d}b)}{\int_{\mathcal{X(K)}}L_n(b)\Pi(\mathrm{d}b)},
\end{equation*}
where
\begin{equation*}
L_n(b)=\frac{\pi_b(X_0)}{\pi_{b_0}(X_0)}\prod_{i=1}^n \frac{p_b(\Delta,X_{(i-1)\Delta},X_{i\Delta})}{p_{b_0}(\Delta,X_{(i-1)\Delta},X_{i\Delta})}
\end{equation*}
is the likelihood ratio. By Lemma \ref{lemma5.3} in Appendix \ref{appendix}, in order to prove the theorem, it suffices to show that
\begin{equation*}
\Pi(B_j^+|X_0,\ldots,X_{n\Delta}) \rightarrow 0, \quad \Pi(B_j^-|X_0,\ldots,X_{n\Delta}) \rightarrow 0, \quad \text{$P_{b_0}$-a.s.}
\end{equation*}
for the sets $B_j^+$ and $B_j^-$ ($j=1,\ldots,N$ for some suitable integer $N>0$) given in the statement of that lemma. We give a brief outline of the remaining part of the proof: thanks to property \eqref{priorcondition} of the prior, by Lemma \ref{lemma5.1} from Appendix \ref{appendix} the prior $\Pi$ has the Kullback-Leibler property in the sense that \eqref{KLproperty} holds. Then by Lemma \ref{lemma5.2} from Appendix \ref{appendix} in order to establish posterior consistency, it suffices to show that $P_{b_0}$-a.s.\ the terms
\begin{equation*}
\sqrt{\int_{B_j^+}L_n(b)\Pi(\mathrm{d}b)}, \quad \sqrt{\int_{B_j^-}L_n(b)\Pi(\mathrm{d}b)},
\end{equation*}
converge to zero exponentially fast. This fact can be proved by the same reasoning as given in the proof of Theorem 3.5 in \cite{meulen} (employing the convergence theorem for a positive supermartingale (see e.g.\ Theorem 22 on p.\ 148 in \cite{pollard}) instead of Doob's martingale convergence theorem on pp.\ 59--60 there\footnote{Note that on p.\ 58 in \cite{meulen} the expression $L_n$ is called the likelihood, although obviously the likelihood ratio is meant.}). This completes the proof.
\end{proof}

\begin{proof}[Proof of Lemma \ref{lemma3.2b}]
The lemma can be proved by arguments similar to those in the proof of Lemma 3.2 in \cite{meulen}. The proof employs Lemma \ref{lemma3.1b} from Appendix \ref{appendixB}, that plays the role of Lemma 3.1 from \cite{meulen} in this context.
\end{proof}

\begin{proof}[Proof of Theorem \ref{mainthm_b}]
The proof is an easy generalisation of the proof of Theorem \ref{mainthm} and  uses lemmas from Appendix \ref{appendixB} instead of Lemmas \ref{lemma5.1}, \ref{lemma5.2} and \ref{lemma5.3} from Appendix \ref{appendix}.
\end{proof}

In the next lemma we verify the claim made at the end of Example \ref{exam_prior2}.

\begin{lemma}
\label{lemmapriorB}
The prior $\Pi$ from Example \ref{exam_prior2} satisfies the requirement \eqref{priorcondition_b}.
\end{lemma}
\begin{proof}
The proof is similar to a demonstration of an analogous property of the prior in Example 4.1 in \cite{meulen}: for every $b\in\mathcal{X}(K_1,K_2)$ and positive integer $m$ we have
\begin{align*}
\sum_{i=1}^d\|b_i-b_{i,0}\|_{2,\mu_{b_0}}^2&= \sum_{i=1}^d \int_{\|x\| \leq m} (b_i(x)-b_{0,i}(x))^2\pi_{b_0}(x)\mathrm{d}x\\
&+\sum_{i=1}^d \int_{\|x\| > m} (b_i(x)-b_{0,i}(x))^2\pi_{b_0}(x)\mathrm{d}x\\
& \leq d \| b-b_{0} \|_{m,d,\infty}\\
&+4K^2d\int_{\|x\| > m}(1+\|x\|)^2\pi_{b_0}(x)\mathrm{d}x.
\end{align*}
Thanks to the fact that $\mu_{b_0}$ has an exponential moment, the second term on the right-hand side can be made less than $\varepsilon^2$ by choosing $m$ large enough. Hence
\begin{multline*}
\Pi\left( b\in\mathcal{X}(K_1,K_2): \sum_{i=1}^d \|  b_i - b_{0,i} \|_{2,\mu_{b_0}}^2 < 2\varepsilon^2 \right)\\
\geq \Pi\left(  b \in\mathcal{X}(K_1,K_2) : \|b-b_0\|_{m,d,\infty}^2<\frac{\varepsilon^2}{d} \right).
\end{multline*}
For $l$ such that $\epsilon_l<\varepsilon/\sqrt{d},$ we have by construction of $\Pi$ that the right-hand side of the above display is bounded from  below by $q_{m,1}q_{l,2}/k_{m,l}>0.$ This completes the proof of the lemma.
\end{proof}

\appendix
\section{  }
\label{appendix}

The following result is a restatement of Lemma 5.1 in \cite{meulen}.

\begin{appxlem}
\label{lemma5.1}
Let
\begin{equation*}
{\rm{KL}}(b_0,b)=\int_{\mathbb{R}}\int_{\mathbb{R}} \pi_{b_0}(x) p_{b_0}(\Delta,x,y)\log\frac{p_{b_0}(\Delta,x,y)}{p_b(\Delta,x,y)}\mathrm{d}x\mathrm{d}y.
\end{equation*}
Then 
for the prior $\Pi$ satisfying property \eqref{priorcondition}, the inequality
\begin{equation}
\label{KLproperty}
\Pi(b\in{\mathcal{X}}(K):{\rm{KL}}(b_0,b)<\varepsilon)>0, \quad \forall \varepsilon>0
\end{equation}
holds.
\end{appxlem}

\begin{proof}
The same proof as in \cite{meulen} goes through. The only additional clarification we would like to make concerns finiteness of ${\rm{KL}}(b_0,b).$ The latter follows from the inequality in the proof of Lemma 5.1 in \cite{meulen},
\begin{equation*}
{\rm{KL}}(b_0,b) \leq - {\rm{K}}(\mu_{b_0},\mu_b)+{\rm{K}}(\mathcal{L}_1,\mathcal{L}_2),
\end{equation*}
where
\begin{multline}
\label{nomeri2}
{\rm{K}}(\mathcal{L}_1,\mathcal{L}_2)\\=\ex_{{P}_{b_0}}\left[ \log\frac{\pi_{b_0}(X_0)}{\pi_{b}(X_0)}-\int_0^{\Delta} (b(X_s)-b_0(X_s))dW_s+\frac{1}{2}\int_0^{\Delta} (b(X_s)-b_0(X_s))^2\mathrm{d}s \right]\\
={\rm{K}}(\mu_{b_0},\mu_{b})+\frac{\Delta}{2}\|b-b_0\|_{2,\mu_{b_0}}^2
\end{multline}
is the Kullback-Leibler divergence between the laws $\mathcal{L}_1=\mathcal{L}(\{X_t,t\in[0,\Delta]\}|{P}_{b_0})$ and $\mathcal{L}_2=\mathcal{L}(X_t,t\in[0,\Delta]|{P}_{b})$ of the full path $\{X_t,t\in[0,\Delta]\}$ under $P_{b_0}$ and $P_b,$ respectively, while ${\rm{K}}(\mu_{b_0},\mu_b)$ is the Kullback-Leibler divergence between the two invariant measures $\mu_{b_0}$ and $\mu_b.$ The second term on the right-hand side of the last equality in \eqref{nomeri2} is finite by the exponential decay property of $\pi_{b_0},$ cf.\ formula \eqref{expfast}. Furthermore, we have
\begin{align*}
{\rm{K}}(\mu_{b_0},\mu_{b})&=\int_{\mathbb{R}}\pi_{b_0}(x)\log\frac{\pi_{b_0}(x)}{\pi_{b}(x)}\mathrm{d}x\\
&\leq\left|\log\frac{m_{b}(\mathbb{R})}{m_{b_0}(\mathbb{R})}\right|+4K\int_{\mathbb{R}}\left(|x|+\frac{x^2}{2}\right)\pi_{b_0}(x)\mathrm{d}x.
\end{align*}
This implies finiteness of the first term on the right-hand side of the last equality in \eqref{nomeri2}, and hence of ${\rm{KL}}(b_0,b)$ too.
\end{proof}

The next lemma is a restatement of Lemma 5.2 in \cite{meulen}, cf.\ also the proof of formula (2.1) in \cite{tang}.

\begin{appxlem}
\label{lemma5.2}
Suppose that a prior $\Pi$ has property \eqref{KLproperty}. If for a sequence $C_n$ of measurable subsets of ${\mathcal{X}}(K)$ there exists a constant $c>0,$ such that
\begin{equation*}
e^{nc}\int_{C_n}L_n(b)\Pi(\mathrm{d}b)\rightarrow 0, \quad \text{$P_{b_0}$-a.s.},
\end{equation*}
then
\begin{equation}
\Pi(C_n|X_0,\ldots,X_{\Delta n})\rightarrow 0, \quad \text{$P_{b_0}$-a.s.}
\end{equation}
as $n\rightarrow\infty.$
\end{appxlem}

\begin{proof}
The proof of Lemma 5.2 in \cite{meulen} remains applicable. The required version of the strong law of large numbers for ergodic sequences invoked in \cite{meulen} follows for instance from Theorems 3.5.8 and 3.5.7 in \cite{stout}.
\end{proof}

The next lemma is a restatement of Lemma 5.3 in \cite{meulen}, the proof of which employs some arguments from \cite{tang}.

\begin{appxlem}
\label{lemma5.3}
Fix $\varepsilon>0$ such that $\varepsilon\leq 2\nu(\mathbb{R}),$ take a fixed $f\in\mathcal{C}_{bdd}(\mathbb{R})$ such that $\|f\|_{\infty}\leq 1,$ and write
\begin{equation*}
B=\{ b\in\mathcal{X}(K):\| P_{\Delta}^b f - P_{\Delta}^{b_0} f \|_{1,\nu} > \varepsilon \}.
\end{equation*}
Then there exist a compact set $F\subset\mathbb{R},$ an integer $N>0$ and bounded intervals $I_1,\ldots,I_N$ covering $F,$ such that
\begin{equation*}
B\subset \left(\bigcup_{j=1}^N B_j^+\right) \bigcup\left( \bigcup_{j=1}^N B_j^-\right),
\end{equation*}
where
\begin{align*}
B_j^+&= \left\{ b\in B : P_{\Delta}^b f(x) - P_{\Delta}^{b_0} f(x) > \frac{\varepsilon}{4\nu(F)},\forall x\in I_j \right\} ,\\
B_j^-&= \left\{ b\in B : P_{\Delta}^b f(x) - P_{\Delta}^{b_0} f(x) < -\frac{\varepsilon}{4\nu(F)},\forall x\in I_j \right\}.
\end{align*}
\end{appxlem}

\begin{proof}
The proof of this lemma requires Lemma \ref{lemma6.1} that corresponds to Lemma A.1 in \cite{meulen}. The arguments from the proof of Lemma 5.3 in \cite{meulen} carry over.
\end{proof}

The next lemma is an adaptation of Lemma A.1 in \cite{meulen}, but in its proof we need somewhat different arguments than those used in \cite{meulen}.

\begin{appxlem}
\label{lemma6.1}
For every fixed $f\in\mathcal{C}_{bdd}(\mathbb{R}),$ the family $\{P_{\Delta}^bf:b\in{\mathcal{X}}(K)\}$ is locally uniformly equicontinuous.
\end{appxlem}

\begin{proof}
By definition we need to show that the family $\{P_{\Delta}^bf:b\in{\mathcal{X}}(K)\}$ is uniformly equicontinuous whenever the argument $x$ of $P_{\Delta}^bf(x)$ is restricted to an arbitrary compact set $F.$
Since the transition operators form a semigroup, it is enough to prove the latter claim for $\Delta$ small enough, in particular for $\Delta$ satisfying
\begin{equation}
\label{conditionK}
K\Delta<\frac{1}{2}.
\end{equation}
In fact, we have
\begin{equation*}
| P_{\Delta}^b f(x) - P_{\Delta}^b f(y) | \leq P_{\Delta/2}^b | P_{\Delta/2}^b f(x) - P_{\Delta/2}^b f(y) |,
\end{equation*}
and if $\{P_{\Delta/2}^bf:b\in{\mathcal{X}}(K)\}$ is uniformly equicontinuous when the argument $x$ ranges in $F,$ then it is immediately seen that so is $\{P_{\Delta}^bf:b\in{\mathcal{X}}(K)\},$ while if not, then we can reiterate the same argument, but now with $\Delta/2$ and $\Delta/4$ instead of $\Delta$ and $\Delta/2$ and so on, until \eqref{conditionK} is met.

Fix a compact set $F$ and let
\begin{equation*}
l_u=\int_0^{\Delta} b(u+W_s)dW_s-\frac{1}{2}\int_0^{\Delta} b^2(u+W_s)\mathrm{d}s, \quad L_u=e^{l_u}
\end{equation*}
for a standard Brownian motion $W.$ Then as in \cite{meulen}, it can be shown that
\begin{equation*}
P_{\Delta}^bf(x)=\ex [f(x+W_{\Delta})L_x],
\end{equation*}
where the expectation is evaluated under the Wiener measure (the Girsanov theorem invoked in \cite{meulen} is applicable in our case thanks to the linear growth condition and Corollary 5.16 on p.\ 200 in \cite{karatzas}). Also
\begin{align*}
| P_{\Delta}^bf(x) - P_{\Delta}^bf(y) | & \leq \ex[|f(x+W_{\Delta})||L_x-L_y|]+\ex[ L_y|f(x+W_{\Delta})-f(y+W_{\Delta})| ]\\
&:=S_1+S_2,
\end{align*}
where $x,y\in F$. We will bound the two terms $S_1$ and $S_2$ separately.

By \eqref{conditionK} there exists $\widetilde{q}>1,$ such that
\begin{equation}
\label{kd}
K\Delta<\frac{1}{2\sqrt{\widetilde{q}}}.
\end{equation}
Fix such $\widetilde{q}$ and let $q$ be determined as that root of the equation
\begin{equation}
\label{qtilde}
\widetilde{q}=2 \left( q^2 - \frac{q}{2} \right)
\end{equation}
that is larger than $1.$ Next set $r=q/(q-1).$ Note that $r>1$ and that $1/r+1/q=1.$

To bound $S_1,$ we apply an elementary inequality $|e^a-e^b|\leq |a-b||e^a+e^b|$ for $a,b\in\mathbb{R}$ and H\"older's inequality with exponents $r$ and $q$ defined as above to obtain
\begin{align*}
S_1 & \leq \|f\|_{\infty}  \ex[|L_x-L_y|]  \\
&\leq \|f\|_{\infty} \ex[|l_x-l_y||L_x+L_y|]\\
& \leq \|f\|_{\infty} \{\ex[|l_x-l_y|^r]\}^{1/r}\{\ex[|L_x+L_y|^{q}]\}^{1/q}.
\end{align*}
In order to bound $S_1,$ we hence need to bound the last two factors on the right-hand side of the last inequality in the above display. We first treat the first of these two. Note that
\begin{equation}
\label{eq6.1}
\begin{split}
l_x-l_y&=\int_0^{\Delta} (b(x+W_s) - b(y+W_s) )dW_s - \frac{1}{2} \int_0^{\Delta} (b^2(x+W_s) - b^2(y+W_s) )\mathrm{d}s\\
&:=S_3+S_4.
\end{split}
\end{equation}
The $c_r$-inequality yields that in order to bound $\{\ex[|l_x-l_y|^r]\}^{1/r},$ it is enough to bound $\ex[|S_3|^r]$ and $\ex[|S_4|^r].$ We bound the first of these two expectations as follows: by the Burkholder-Davis-Gundy inequality, see Theorem 3.28 on p.\ 166 in \cite{karatzas},
\begin{multline*}
\ex\left[ \left| \int_0^{\Delta} (b(x+W_s) - b(y+W_s) )dW_s \right|^r \right] \\
\leq C_r \ex\left[ \left| \int_0^{\Delta} (b(x+W_s) - b(y+W_s) )^2 \mathrm{d}s \right|^{r/2} \right],
\end{multline*}
where $C_r>0$ is a universal constant independent of $b,x$ and $y.$ For a constant $R>0$ and the set $F^{\prime}=\{ u+v:u\in F,v\in [-R,R]\}$ by the Cauchy-Schwarz inequality the expectation on the right-hand side of the above display can be handled as follows:
\begin{multline*}
\ex\left[ \left| \int_0^{\Delta} (b(x+W_s) - b(y+W_s) )^2  \mathrm{d}s\right|^{r/2} 1_{ [\sup_{s \leq {\Delta}} |W_s| \leq R ] } \right]\\
+\ex\left[ \left| \int_0^{\Delta} (b(x+W_s) - b(y+W_s) )^2  \mathrm{d}s\right|^{r/2} 1_{ [\sup_{s \leq {\Delta}} |W_s| > R ] } \right]\\
\leq {\Delta}^{r/2} \sup_{ \substack{ u,v\in F^{\prime} \\ |u-v| \leq |x-y| } } |b(u)-b(v)|^r\\
+\left\{ \ex\left[ \left| \int_0^{\Delta} (b(x+W_s) - b(y+W_s) )^2  \mathrm{d}s \right|^{r} \right] \right\}^{1/2} \left\{ P\left( \sup_{s \leq {\Delta}} |W_s| > R \right)\right\}^{1/2}.
\end{multline*}
Thanks to the fact that $\mathcal{X}(K)$ is a locally uniformly equicontinuous family of functions, for a fixed $R$ the first term on the right-hand side of the last inequality can be made arbitrarily small uniformly in $b\in\mathcal{X}(K)$ by choosing $\delta$ small enough and  $|x-y|\leq\delta.$ Also $P\left( \sup_{s \leq {\Delta}} |W_s| > R \right)$ can be made arbitrarily small by choosing $R$ large enough. Finally,
\begin{multline*}
\ex\left[ \left| \int_0^{\Delta} (b(x+W_s) - b(y+W_s) )^2  \mathrm{d}s \right|^{r} \right]\\
\leq {\Delta}^{r/q}  \ex\left[ \int_0^{\Delta} |b(x+W_s) - b(y+W_s) |^{2r} \mathrm{d}s \right].
\end{multline*}
The expectation on the right-hand side is bounded by a universal constant independent of particular $x,y\in F$ and $b\in\mathcal{X}(K).$ This can be seen by a simple, but lengthy computation employing the Fubini theorem, the linear growth condition on $b,$ the $c_{2r}$-inequality and the fact that $W$ has moments of all orders. This completes bounding $\ex[|S_3|^r],$ which hence can be made arbitrarily small uniformly in $b\in\mathcal{X}(K),$ once $|x-y|\leq\delta$ for $\delta>0$ small enough. The term $\ex[|S_4|^r]$ can be bounded using similar computations by employing the inequality
\begin{equation*}
|b^2(u)-b^2(v)| \leq K(2+|u|+|v|) |b(u)-b(v)|
\end{equation*}
and the Cauchy-Schwarz inequality (twice), yielding together
\begin{multline*}
\ex\left[ \left| \int_0^{\Delta} ( b^2(x+W_s)-b^2(y+W_s))\mathrm{d}s \right|^r \right]\\
\leq \left\{ \ex\left[ \left| \int_0^{\Delta} ( b(x+W_s)-b(y+W_s))^2\mathrm{d}s \right|^r \right] \right\}^{1/2}\\
\times \left\{ \ex\left[ \left| \int_0^{\Delta} ( b(x+W_s)+b(y+W_s))^2\mathrm{d}s \right|^r \right] \right\}^{1/2}.
\end{multline*}
The first factor on the right-hand side can be made arbitrarily small uniformly in $b\in\mathcal{X}(K)$ by taking $\delta$ small (cf.\ above), while the second factor remains bounded uniformly in $b\in\mathcal{X}(K)$ and $x,y,\in F.$ To complete bounding $S_1,$ we need to bound the right-hand side of the inequality
\begin{equation*}
\ex[|L_x+L_y|^{q}] \leq c_{q} \ex[L_x^{q}] + c_{q} \ex[L_y^{q}].
\end{equation*}
Since obviously both terms on the right-hand side can be bounded in exactly the same manner, we will only give an argument for one of them. By the Cauchy-Schwarz inequality applied to two random variables
\begin{gather*}
\exp\left( \left( q^2 - \frac{q}{2} \right) \int_0^{\Delta} b^2 (x+W_s)\mathrm{d}s \right),\\
\quad \exp\left( \int_0^{\Delta} q b (x+W_s)dW_s -  \int_0^{\Delta} q^2 b^2 (x+W_s)\mathrm{d}s \right),
\end{gather*}
we have
\begin{equation}
\label{Lu}
\ex[L_x^q] \leq \left\{ \ex\left[ \exp\left( 2 \left( q^2 - \frac{q}{2} \right) \int_0^{\Delta} b^2 (x+W_s)\mathrm{d}s \right) \right]\right\}^{1/2}.
\end{equation}
Here we used the fact that
\begin{equation}
\label{expmartingale}
\ex\left[ \exp\left( \int_0^{\Delta} 2 q b(x+W_s)dW_s - \frac{1}{2} \int_0^{\Delta} 4q^2 b^2 (x+W_s)\mathrm{d}s \right) \right]=1,
\end{equation}
since the process under the expectation sign is a martingale and has expectation equal to one (this is due to the linear growth condition and Corollary 5.16 on p.\ 200 in \cite{karatzas}). Hence it remains to bound the right-hand side of \eqref{Lu}, which we denote by $S_5.$ 
By the linear growth condition we have
\begin{equation*}
S_5^2 \leq \exp\left( 2 \widetilde{q}K^2{\Delta}(1+|x|)^2 \right) \ex\left[ \exp\left( 2 \widetilde{q} K^2 \int_0^{\Delta} W^2_s \mathrm{d}s \right) \right].
\end{equation*}
Showing finiteness of the expectation on the right-hand side is standard: by Doob's maximal inequality for submartingales (see Theorem 3.8 (iv) pp.\ 13--14 in \cite{karatzas}; that the exponential on the right-hand side of the first line of the displayed formula below is a submartingale follows from Problem 3.7 on p.\ 13 in \cite{karatzas}),
\begin{align*}
\ex\left[ \exp\left( 2 \widetilde{q} K^2 \int_0^{\Delta} W^2_s\mathrm{d}s \right) \right] & \leq\ex \left[ \sup_{s \leq {\Delta}} \exp \left( 2 \widetilde{q} K^2 {\Delta} W_s^2 \right) \right]\\
&\leq 4 \ex \left[ \exp \left( 2 \widetilde{q} K^2 {\Delta} W_{\Delta}^2 \right) \right]<\infty.
\end{align*}
Here in the last inequality we used \eqref{kd}.

A conclusion that follows from the above bounds is that $S_1$ can be made arbitrarily small as soon as $|x-y|\leq\delta$ for small enough $\delta.$ The bound on $S_1$ will be true uniformly in $b\in\mathcal{X}(K).$

In order to bound $S_2,$ we again use H\"older's inequality to get
\begin{equation*}
S_2 \leq \{\ex[L_y^{q}]\}^{1/q} \{ \ex[|f(x+W_{\Delta})-f(y+W_{\Delta})|^r] \}^{1/r}.
\end{equation*}
The first factor on the right-hand side can be bounded as above. The second factor can be made arbitrarily small as soon as $|x-y|\leq\delta$ for small enough $\delta.$ Indeed, for a constant $R>0$ write
\begin{align*}
\ex[|f(x+W_{\Delta})-f(y+W_{\Delta})|^r]&=\ex[|f(x+W_{\Delta})-f(y+W_{\Delta})|^r 1_{[|W_{\Delta}|>R]}]\\
&+\ex[|f(x+W_{\Delta})-f(y+W_{\Delta})|^r 1_{[ |W_{\Delta}| \leq R]}]\\
&\leq (2\|f\|_{\infty})^r P(|W_{\Delta}|>R)\\
&+\ex[|f(x+W_{\Delta})-f(y+W_{\Delta})|^r 1_{[ |W_{\Delta}| \leq R]}].
\end{align*}
It is obvious that the first term on the right-hand side of the last inequality can be made arbitrarily small by selecting $R$ large enough. However, so can be the second one upon fixing $R$ by taking $|x-y|\leq \delta$ for small enough $\delta>0,$ since the function $f$ is uniformly continuous on compacts. This completes the proof.
\end{proof}

\section{}
\label{appendixB}

\begin{appxlem}
\label{lemma3.1b} Let $b,\widetilde{b}\in\widetilde{\mathcal{X}}(K_1,K_2).$ Fix $t>0.$ If $b\neq\widetilde{b},$ then $P_t^{b} \neq P_t^{\widetilde{b}}.$
\end{appxlem}

\begin{proof}
The proof is similar to the proof of Lemma 3.1 in \cite{meulen}. By continuity of $b$ and $\widetilde{b}$ we have that if $b\neq\widetilde{b},$ this in fact holds on a set of positive Lebesgue measure. Then also $V_b\neq V_{\widetilde{b}}$ on a set of positive Lebesgue measure and therefore $\pi_b\neq \pi_{\widetilde{b}}$ on a set of positive Lebesgue measure, for instance some open ball in $\mathbb{R}^d.$ Now assume that $P_{t}^b=P_{t}^{\widetilde{b}}.$ Then for any bounded measurable function $f$ and any positive integer $m,$  by the semigroup property of $P_{t}^b$ we have that
\begin{equation*}
\ex_x^b[f(X_{mt})]=(P_t^b)^m f(x)=(P_t^{\widetilde{b}})^m f(x)=\ex_x^{\widetilde{b}}[f(X_{mt})].
\end{equation*}
Letting $m\rightarrow\infty,$ the above display and ergodicity give that
\begin{equation*}
\int_{\mathbb{R}^d}f(y)\pi_b(y)\mathrm{d}y=\int_{\mathbb{R}^d}f(y)\pi_{\widetilde{b}}(y)\mathrm{d}y.
\end{equation*}
Hence $\pi_b=\pi_{\widetilde{b}}$ a.e., and in fact by continuity $\pi_b=\pi_{\widetilde{b}}$ everywhere. This is a contradiction and thus $b\neq\widetilde{b}$ implies $P_{t}^b\neq P_{t}^{\widetilde{b}}.$
\end{proof}

\begin{appxlem}
\label{lemma5.3b}
Fix $\varepsilon>0$ such that $\varepsilon\leq 2\nu(\mathbb{R}),$ take a fixed $f\in\mathcal{C}_{bdd}(\mathbb{R}^d)$ such that $\|f\|_{\infty}\leq 1,$ and write
\begin{equation*}
B=\{ b\in\mathcal{X}:\| P_{\Delta}^b f - P_{\Delta}^{b_0} f \|_{1,\nu} > \varepsilon \}.
\end{equation*}
Then there exist a compact set $F\subset\mathbb{R}^d,$ an integer $N>0$ and cubes $I_1,\ldots,I_N$ covering $F,$ such that
\begin{equation*}
B\subset \left(\bigcup_{j=1}^N B_j^+\right) \bigcup\left( \bigcup_{j=1}^N B_j^-\right),
\end{equation*}
where
\begin{align*}
B_j^+&= \left\{ b\in B : P_{\Delta}^b f(x) - P_{\Delta}^{b_0} f(x) > \frac{\varepsilon}{4\nu(F)},\forall x\in I_j \right\} ,\\
B_j^-&= \left\{ b\in B : P_{\Delta}^b f(x) - P_{\Delta}^{b_0} f(x) < -\frac{\varepsilon}{4\nu(F)},\forall x\in I_j \right\}.
\end{align*}
\end{appxlem}

\begin{proof}
The proof of Lemma 5.3 in \cite{meulen} carries over, provided one redefines the intervals $I_j$ of length $\delta/2>0$ from that proof to be cubes with sides of length $\delta/2,$ and uses instead of Lemma A.1 from \cite{meulen} Lemma \ref{lemma6.1b} below.
\end{proof}

\begin{appxlem}
\label{lemma6.1b}
For a fixed $f\in\mathcal{C}_{bdd}(\mathbb{R}^d)$ and $t>0,$ the family $\{P_t^bf:b\in {\mathcal{X}(K_1,K_2)} \}$ is a locally uniformly equicontinuous family of functions.
\end{appxlem}

\begin{proof}
Let
\begin{equation*}
l_u=\sum_{i=1}^d\int_0^{\Delta} b_i (u+W_s)\mathrm{d}W_{i,s}-\frac{1}{2}\sum_{i=1}^d\int_0^{\Delta} b^2_i(u+W_{s})\mathrm{d}s, \quad L_u=e^{l_u}
\end{equation*}
for a standard $d$-dimensional Brownian motion $W=(W_1,\ldots,W_d).$ Then, employing the Girsanov theorem, as in the proof of Lemma A.1 in \cite{meulen}, see also the proof of Lemma \ref{lemma6.1}, it can be shown that
\begin{equation*}
P_{\Delta}^bf(x)=\ex [f(x+W_{\Delta})L_x],
\end{equation*}
where the expectation is evaluated under the Wiener measure. From this point on  the proof is a generalisation of the arguments in the proof of Lemma \ref{lemma6.1} from Appendix \ref{appendix} to the multidimensional setting. In particular, as in that proof, it is enough to prove the lemma for $\Delta$ such that
\begin{equation*}
\Delta K_1<\frac{1}{2\sqrt{d}}.
\end{equation*}
In order to prove the lemma, we need to show that the family of functions $\{P_t^bf:b\in {\mathcal{X}(K_1,K_2)} \}$ is uniformly equicontinuous whenever the argument $x$ of $P_t^bf(x)$ is restricted to an arbitray compact set $F.$  Fix a compact set $F\subset\mathbb{R}^d.$ Throughout this proof we assume $x,y\in F.$ We have
\begin{align*}
| P_{\Delta}^bf(x) - P_{\Delta}^bf(y) | & \leq \ex[|f(x+W_{\Delta})||L_x-L_y|]+\ex[ L_y|f(x+W_{\Delta})-f(y+W_{\Delta})| ]\\
&:=S_1+S_2.
\end{align*}
We will bound the two terms $S_1$ and $S_2$ separately. There exists $\widetilde{q}>1,$ such that
\begin{equation}
\label{kd_b}
K_1\Delta<\frac{1}{2\sqrt{d\widetilde{q}}}.
\end{equation}
Fix such $\widetilde{q}$ and let $q$ be determined as that root of the equation
\begin{equation}
\label{qtilde_b}
\widetilde{q}=2 \left( q^2 - \frac{q}{2} \right)
\end{equation}
that is larger than $1.$ Next set $r=q/(q-1).$ We first bound $S_1.$ As in the proof of Lemma \ref{lemma6.1} from Appendix \ref{appendix}, we have that
\begin{equation*}
S_1 \leq \|f\|_{\infty} \{\ex[|l_x-l_y|^r]\}^{1/r}\{\ex[|L_x+L_y|^{q}]\}^{1/q}.
\end{equation*}
The $c_r$-inequality gives that in order to bound $ \{\ex[|l_x-l_y|^r]\}^{1/r}, $ it is enough to bound the terms
\begin{gather*}
\ex\left[ \left| \int_0^{\Delta} ( b_i(x+W_s) - b_i(y+W_s) )\mathrm{d}W_{i,s} \right|^{r} \right], \\
\ex\left[ \left| \int_0^{\Delta} ( b_i^2(x+W_s) - b_i^2(y+W_s) )\mathrm{d}s \right|^{r} \right]
\end{gather*}
for $i=1,\ldots,d.$ Since the arguments are the same for any $i,$ we henceforth fix a particular $i.$ As in Lemma \ref{lemma6.1}, the Burkholder-Davis-Gundy inequality gives
\begin{multline*}
\ex\left[ \left| \int_0^{\Delta} (b_i(x+W_s) - b_i(y+W_s) )\mathrm{d}W_{i,s} \right|^r \right] \\
\leq C_r \ex\left[ \left| \int_0^{\Delta} (b_i(x+W_s) - b_i(y+W_s) )^2 \mathrm{d}s \right|^{r/2} \right],
\end{multline*}
where $C_r>0$ is a universal constant independent of $b.$ For a constant $R>0$ and the set $F^{\prime}=\{ u+v:u\in F,\|v\|\leq R\}$ by the Cauchy-Schwarz inequality the expectation on the right-hand side of the above display can be bounded as follows:
\begin{multline*}
\ex\left[ \left| \int_0^{\Delta} (b_i(x+W_s) - b_i(y+W_s) )^2  \mathrm{d}s\right|^{r/2} 1_{ [\sup_{s \leq {\Delta}} \|W_s\| \leq R ] } \right]\\
+\ex\left[ \left| \int_0^{\Delta} (b_i(x+W_s) - b_i(y+W_s) )^2  \mathrm{d}s\right|^{r/2} 1_{ [\sup_{s \leq {\Delta}} \|W_s\| > R ] } \right]\\
\leq {\Delta}^{r/2} \sup_{ \substack{ u,v\in F^{\prime} \\ \|u-v\| \leq \|x-y\| } } |b_i(u)-b_i(v)|^r\\
+\left\{ \ex\left[ \left| \int_0^{\Delta} (b_i(x+W_s) - b_i(y+W_s) )^2  \mathrm{d}s \right|^{r} \right] \right\}^{1/2} \left\{ P\left( \sup_{s \leq {\Delta}} \|W_s\| > R \right)\right\}^{1/2}.
\end{multline*}
Since $b$ has partial derivatives bounded in absolute value by $K_2,$ the first term on the right-hand side of the above display can be made arbitrarily small by choosing $\delta$ small enough and $\|x-y\|\leq \delta.$ Furthermore, the term
\begin{equation*}
\left\{ P\left( \sup_{s \leq {\Delta}} \|W_s\| > R \right)\right\}^{1/2}
\end{equation*}
can be made arbitrarily small by choosing $R$ large enough. A lengthy, but easy computation shows that the term
\begin{equation*}
\left\{ \ex\left[ \left| \int_0^{\Delta} (b_i(x+W_s) - b_i(y+W_s) )^2  \mathrm{d}s \right|^{r} \right] \right\}^{1/2}
\end{equation*}
is bounded by a constant independent of $b;$ cf.\ the arguments in the proof of Lemma \ref{lemma6.1} from Appendix \ref{appendix}. Consequently, the term
\begin{equation*}
\ex\left[ \left| \int_0^{\Delta} ( b_i(x+W_s) - b_i(y+W_s) )\mathrm{d}W_{i,s} \right|^{r} \right]
\end{equation*}
can be made arbitrarily small, once $\delta$ is chosen small enough and $\|x-y\|\leq\delta.$ The term
\begin{equation*}
\ex\left[ \left| \int_0^{\Delta} ( b_i^2(x+W_s) - b_i^2(y+W_s) )\mathrm{d}s \right|^{r} \right]
\end{equation*}
can be shown to be bounded uniformly in $b\in\mathcal{X}(K_1,K_2)$ by employing similar techniques; cf.\ the proof of Lemma \ref{lemma6.1} from Appendix \ref{appendix}. Next we need to bound the right-hand side of the inequality
\begin{equation*}
\ex[|L_x+L_y|^{q}] \leq c_{q} \ex[L_x^{q}] + c_{q} \ex[L_y^{q}].
\end{equation*}
Since obviously both terms on the right-hand side can be bounded in exactly the same manner, we will only give an argument for the first one of them. By the Cauchy-Schwarz inequality applied to the random variables
\begin{gather*}
\exp\left( \left( q^2 - \frac{q}{2} \right) \sum_{i=1}^d \int_0^{\Delta} b_i^2 (x+W_s)\mathrm{d}s \right),\\
\quad \exp\left( \sum_{i=1}^d \int_0^{\Delta} q b_i (x+W_s)\mathrm{d}W_{i,s} - \sum_{i=1}^d  \int_0^{\Delta} q^2 b_i^2 (x+W_s)\mathrm{d}s \right),
\end{gather*}
as in the proof of Lemma \ref{lemma6.1} from Appendix \ref{appendix} we have
\begin{equation}
\label{Lu_b}
\ex[L_x^q] \leq \left\{ \ex\left[ \exp\left( 2 \left( q^2 - \frac{q}{2} \right)  \sum_{i=1}^d \int_0^{\Delta} b_i^2 (x+W_s)\mathrm{d}s \right) \right]\right\}^{1/2}.
\end{equation}
Hence it remains to bound the right-hand side of the above display, which we denote by $S_5.$ 
By the linear growth condition we have
\begin{equation*}
S_5^2 \leq \exp\left( 2 d \widetilde{q}K_1^2{\Delta} (1+\|x\|)^2 \right) \ex\left[ \exp\left( 2 d \widetilde{q} K_1^2 \int_0^{\Delta} \|W_s\|^2 \mathrm{d}s \right) \right].
\end{equation*}
By Doob's maximal inequality for submartingales and independence of scalar Brownian motions $W_i$'s,
\begin{equation*}
\ex\left[ \exp\left( 2 d \widetilde{q} K_1^2 \int_0^{\Delta} \|W_s\|^2\mathrm{d}s \right) \right] \leq 4 \prod_{i=1}^d \ex \left[ \exp \left( 2 d \widetilde{q} K_1^2 {\Delta} W_{i,\Delta}^2 \right) \right]<\infty.
\end{equation*}
Here in the last inequality we used \eqref{kd_b}. The conclusion is that the term $S_1$ can be made arbitrarily small by taking $\delta$ small and $\|x-y\|\leq\delta.$ The proof is now completed as in the case of Lemma \ref{lemma6.1} from Appendix \ref{appendix}: by H\"older's inequality
\begin{equation*}
S_2 \leq \{\ex[L_y^{q}]\}^{1/q} \{ \ex[|f(x+W_{\Delta})-f(y+W_{\Delta})|^r] \}^{1/r}.
\end{equation*}
The first factor on the right-hand side can be bounded as above uniformly in $b\in\mathcal{X}(K_1,K_2).$ The second factor can be made arbitrarily small as soon as $\|x-y\|\leq\delta$ for small enough $\delta$: for a constant $R>0,$
\begin{align*}
\ex[|f(x+W_{\Delta})-f(y+W_{\Delta})|^r]&=\ex[|f(x+W_{\Delta})-f(y+W_{\Delta})|^r 1_{[\|W_{\Delta}\|>R]}]\\
&+\ex[|f(x+W_{\Delta})-f(y+W_{\Delta})|^r 1_{[ \|W_{\Delta}\| \leq R]}]\\
&\leq (2\|f\|_{\infty})^r P(\|W_{\Delta}\|>R)\\
&+\ex[|f(x+W_{\Delta})-f(y+W_{\Delta})|^r 1_{[ \|W_{\Delta}\| \leq R]}].
\end{align*}
The first term on the right-hand side of the last inequality can be made arbitrarily small by selecting $R$ large enough. Upon fixing $R,$ so can be the second one by taking $\|x-y\|\leq \delta$ for small enough $\delta>0.$ Combination of all the above intermediate results entails the statement of the lemma.
\end{proof}

\begin{appxlem}
\label{lemma5.1b}
Let
\begin{equation*}
{\rm{KL}}(b_0,b)=\int_{\mathbb{R}^d}\int_{\mathbb{R}^d} \pi_{b_0}(x) p_{b_0}(\Delta,x,y)\log\frac{p_{b_0}(\Delta,x,y)}{p_b(\Delta,x,y)}\mathrm{d}x\mathrm{d}y.
\end{equation*}
and assume that the weak solution to \eqref{sde} is initialised at $\mu_b.$ Then
for the prior $\Pi$ satisfying property \eqref{priorcondition_b}, the inequality
\begin{equation}
\label{KLproperty_b}
\Pi(b\in{\mathcal{X}}(K_1,K_2):{\rm{KL}}(b_0,b)<\varepsilon)>0, \quad \forall \varepsilon>0
\end{equation}
holds.
\end{appxlem}

\begin{proof}
The proof is an obvious modification of the proof of Lemma 5.1 in in \cite{meulen}; as in the proof of Lemma \ref{lemma5.1} in Appendix \ref{appendix}, we need to verify additionally that the Kullback-Leibler divergence $\operatorname{K}(\mu_{b},\mu_{\widetilde{b}})$ is finite for any $b,\widetilde{b}\in{\mathcal{X}}(K_1,K_2).$ This, however, follows from Proposition 1.1 in \cite{gobet01}.
\end{proof}

\begin{appxlem}
\label{lemma5.2b}
Suppose that the prior $\Pi$ on $\mathcal{X}(K_1,K_2)$ has the property \eqref{KLproperty_b} and assume that the weak solution to \eqref{sde} is initialised at $\mu_b.$ If for a sequence $C_n$ of measurable subsets of ${\mathcal{X}}(K_1,K_2)$ there exists a constant $c>0,$ such that
\begin{equation*}
e^{nc}\int_{C_n}L_n(b)\Pi(\mathrm{d}b)\rightarrow 0, \quad \text{$P_{{b}_0}$-a.s.},
\end{equation*}
then
\begin{equation*}
\Pi(C_n|X_0,\ldots,X_{\Delta n})\rightarrow 0, \quad \text{$P_{{b}_0}$-a.s.}
\end{equation*}
as $n\rightarrow\infty.$
\end{appxlem}

\begin{proof}
The proof is an easy generalisation of the proof of Lemma 5.2 in \cite{meulen}.
\end{proof}

\bibliographystyle{plainnat}

\end{document}